\documentclass[letterpaper, 10pt, conference]{ieeeconf}      
\IEEEoverridecommandlockouts                              
\usepackage[applemac]{inputenc}
\usepackage{graphics} 
\usepackage{epsfig} 
\usepackage{times} 
\usepackage{amsmath} 
\usepackage{amssymb}  
\usepackage{epstopdf}
\usepackage{graphicx}
\usepackage{caption}
\usepackage{subcaption}
\usepackage{lipsum}
\usepackage{tikz}
\usepackage{mathtools}

\newtheorem{thm}{Theorem}

\newtheorem{lem}{Lemma}
\newtheorem{cor}{Corollary}
\newtheorem{rem}{Remark}
\newtheorem{prop}{Proposition}

\newtheorem{problem}{Problem}

\newcommand{\nt}{{\mathbb N}}
\newcommand{\real}{{\mathbb R}}

\title{\LARGE \bf
Barrier Functionals for Output Functional Estimation of PDEs
}

\author{Mohamadreza Ahmadi, Giorgio Valmorbida, Antonis Papachristodoulou%
\thanks{The authors are with the Department
of Engineering Science, University of Oxford, Oxford,
OX1 3PJ, UK e-mail: (\{mohamadreza.ahmadi, antonis, giorgio.valmorbida\}@eng.ox.ac.uk). M. Ahmadi is supported by the Oxford Clarendon Scholarship and the Sloane-Robinson Scholarship. G. Valmorbida is also affiliated to Somerville College, University of Oxford, Oxford, U.K. A. Papachristodoulou was supported in part by the Engineering and Physical Sciences Research Council projects EP/J012041/1, EP/I031944/1
and EP/J010537/1.
}
}

\begin{document}

\maketitle
\thispagestyle{empty}
\pagestyle{empty}

\begin{abstract}

We propose a method for computing bounds on output functionals of  a class of time-dependent PDEs.  To this end, we introduce barrier functionals for PDE systems. By defining appropriate unsafe sets and optimization problems, we formulate an output functional bound estimation approach based on barrier functionals. In the case of polynomial data, sum of squares (SOS) programming is used to construct the barrier functionals and thus to compute bounds on the output functionals via semidefinite programs (SDPs). An example is given to illustrate the results.

\end{abstract}

\section{INTRODUCTION} \label{sec:intro}

A very large class of systems is described by partial differential equations (PDEs), which include derivatives with respect to both space and time. To name but a few, mechanics of fluid flows~\cite{DG95}, elastic beams~\cite{CA98}, and the magnetic flux profile in a tokamak~\cite{6426638} are all described by PDEs.

In many engineering design problems, one may merely be interested in computing a functional of the solution to the underlying PDE rather than the solution itself (see the review article~\cite{BS94} for a number of applications in structural mechanics).  The far-field pattern in electromagnetics and acoustics  \cite{MS98} and energy release rate in elasticity theory \cite{XPP06} are both functionals of the solutions to
the governing PDEs.

Perhaps a more interesting example is in fluid mechanics, i.e. lift and drag forces acting on an airfoil surrounded by a compressible flow (described by  Euler's equations)  are defined as functionals of pressure and shear forces over the surface of the airfoil~\cite{MPP99}. To illustrate, the dynamics of a compressible flow~\cite{V02} are given by
\begin{equation*}
\partial_t U+{\partial_x F}+{\partial_y G}=0,
\end{equation*}
wherein,
\begin{eqnarray}
U =  \begin{bmatrix}  \rho \\ \rho u \\ \rho v \\  E  \end{bmatrix},~F =  \begin{bmatrix}  \rho u \\ \rho u^2+p \\ \rho uv \\ u( E+p)  \end{bmatrix},~G =  \begin{bmatrix}  \rho v \\ \rho uv \\ \rho v^2+p \\ v( E+p)  \end{bmatrix}. \nonumber
\end{eqnarray}
In the above expressions, $\rho$ is the mass density, $u$ and $v$ are the gas velocities in the $x$ and $y$ directions, $p$ is the static pressure and $E$ is the total energy per unit volume. The relation among $E$, $\rho$, $p$, $u$ and $v$ is given by the ideal gas law 
$$
p= (\gamma-1)\left(E-\rho\frac{u^2+v^2}{2}\right),
$$
where  $\gamma$  is the adiabatic index. Then, the aerodynamic force $\bar{F}_A$ acting on  the airfoil $\Omega$ is given by the functional
$$
\bar{F}_A \equiv \int_{\partial \Omega} p~\bar{n} \,\, \mathrm{d}s,
$$
where  $\bar{n}$ is the unit normal vector to the surface $\partial \Omega$ of the airfoil. The lift $L$ and the drag $D$ forces are functionals defined as
$$
L=|\bar{F}_A|\sin{\theta},~D=|\bar{F}_A|\cos{\theta}
$$
where  $\theta$ is the angle between free stream flow and $\bar{F}_A$ (the angle of attack). Estimation of output functionals such as $L$ and $D$ is a very important problem in aerodynamic design. Most approaches answer this query by computing the solution, and then computing the output functional.

The ubiquity of applications like the one mentioned above has motivated the researchers into developing  computational algorithms for output functional approximation. In \cite{MPP99},  an \textit{a posteriori}
finite element method is proposed
 for estimating lower and upper bounds of output functionals for {semilinear} elliptic PDEs. In~\cite{PP98}, an augmented Lagrangian-based
 approach is proposed for calculation of lower and upper
 bounds to {linear} output functionals of coercive PDEs. In~\cite{MS98}, adjoint and defect methods for obtaining estimates of linear output
 functionals for a class of steady (time-independent) PDEs are suggested. In
 \cite{XLP03}, the authors formulate an  \mbox{\textit{a posteriori}} bound methodology for {linear} output functionals of finite
 element solutions to linear coercive PDEs. Adjoint and defect methods for computing estimates of the error in integral functionals of solutions to steady linear PDEs are discussed in~\cite{Pierce2004769}. In~\cite{Bertsimas2006}, an SDP-based bound estimation approach for linear output functionals of linear elliptic PDEs, based on the moments problem, is formulated.

 However,  most of the methods proposed to date require finite element approximations of the solution, which is susceptible to inherent discretization errors. Also, the computational burden increases as the accuracy of an
 approximated solution is improved. Furthermore, it is not clear whether an attained bound from finite element approximations on the output
 functionals is an upper or lower bound estimate. Consequently, we need certificates to corroborate and verify an obtained bound~(see~\cite{Sauer-budge_computingbounds,XPP06,Pars2006406} for  finite element based methods with certificates for linear/quadratic output functionals of steady linear elliptic PDEs). We show that one approach to certify an obtained bound is through the use of barrier certificates.

 Barrier certificates \cite{P06} were first introduced for model invalidation of ordinary differential equations (ODEs) with polynomial vector fields and have been used to address safety
  verification of nonlinear and hybrid systems~\cite{PJP07},  safety verification
  of a life support system~\cite{GSAGLP07}, and reachability analysis of complex biological networks~\cite{EFLPP06}.
    Moreover, compositional barrier certificates and converse results were studied in \cite{SWP12} and \cite{WC13}, respectively.

 This paper proposes a framework to compute bounds on output functionals of a class of time-dependent PDEs using SDPs, without the need to approximate the solutions. We generalize the result in~\cite{P06} to PDE systems by introducing Barrier Functionals. We show how  different output functionals can be converted into the functional structure suitable for the formulations given in this paper in terms of integral inequalities. The integral inequalities are then solved using the results  in~\cite{VAP14} which have been applied in~\cite{AVP14} for solving dissipation inequalities for PDEs.  For the case of polynomial PDEs and polynomial output functionals (in both dependent and independent variables), SOS programming can be used to construct the barrier functionals and therefore to compute upper bounds. This reduces the problem to solving SDPs.
The proposed upper bound estimation method  is illustrated with an   example.

  The rest of the paper is organized as follows. In the next section, we  give a motivating
  example 
  and formulate the problem under study. In Section~\ref{sec:integinq}, we briefly discuss the method developed in~\cite{VAP14} for studying integral inequalities based on SDPs.
   Section~\ref{sec:barrier} considers the bound estimation method using barrier functionals. In Section~\ref{sec:exam}, we illustrate the proposed results using an example. Finally, Section~\ref{sec:conclusions} concludes the paper and gives directions for future research.


\textbf{Notation:}


The $n$-dimensional Euclidean space is denoted by $\mathbb{R}^n$ and the space of nonnegative reals by $\mathbb{R}_{\ge0}$. The $n$-dimensional space of positive integers is denoted by $\mathbb{N}^n$, and the $n$-dimensional space of non-negative integers is denoted by $\mathbb{N}^n_0$. The set of symmetric $n\times n$ matrices by $\mathbb{S}^n$. 
 The notation $M^\prime$ denotes the transpose of matrix $M$. 
A domain $\Omega$ is a subset of $\mathbb{R}$, and $\overline{\Omega}$ is the closure of set $\Omega$. The boundary $\partial \Omega$ of set $\Omega$ is defined as $\overline{\Omega} \setminus \Omega$ with $\setminus$ denoting set subtraction.
The space of $k$-times continuous differentiable functions defined on $\Omega$ is denoted by $\mathcal{C}^k(\Omega)$. For a multivariable function $f(x,y)$, we use the notation $f\in \mathcal{C}^k[x]$ to show $k$-times continuous differentiability of $f$ with respect to variable $x$.  If $p\in \mathcal{C}^1(\Omega)$, then $\partial_x p$ denotes the derivative of $p$ with respect to variable $x \in \Omega$, i.e. $\partial_x :=\frac{\partial}{\partial x}$.
In addition, we adopt Schwartz's  multi-index notation. For $u \in \mathcal{C}^{\alpha}(\Omega)$, $\alpha \in \nt^n_0$, define
$$
D^{\alpha}u := \left(u_1,\partial_x u_1,  \ldots, {\partial_x^{\alpha_1}{u_{1}}},  \ldots, u_n,\partial_x u_n, \ldots, \partial_x^{\alpha_n} u_n \right).
$$
We denote the ring of polynomials with real coefficients by $\mathcal{R}[x]$,  and the ring of polynomials with a sum-of-squares decomposition by $\Sigma[x] \subset \mathcal{R}[x]$. A polynomial $p(x)\in \Sigma[x]$ if $\exists p_i(x) \in \mathcal{R}[x]$, $i \in \{1, \ldots, n_d\}$ such that $p(x) = \sum_i^{n_d} p_i^2(x)$. Hence, $p(x)$ is clearly non-negative. The set of polynomials~$\{p_i\}_{i=1}^{n_d}$ is called \emph{SOS decomposition} of $p(x)$. The converse does not hold in general, that is, there exist non-negative polynomials which do not have an SOS decomposition~\cite{Par00}.  The test whether  an SOS decomposition exists for a given polynomial can be cast as an SDP (see~\cite{CLR95,Par00,CTVG99}).

\section{Motivating Example and Problem Formulation} \label{sec:mot}

Next, we present a motivating example that is referred to throughout the paper.

\subsection{Motivating Example:}

The heat distribution over a heated rod is described by
\begin{equation} \label{eq:exampsys}
\partial_t u = k\partial_x^2 u + f(t,x,u), \quad x \in \Omega,~t>0
\end{equation}
where $\Omega = [0,1]$, $k>0$ is the thermal conductivity, and $f(t,x,u)$ is the forcing, representing either a heat sink or a heat source.
The initial heat distribution is $u(0,x)=u_0(x)$. We are interested in estimating bounds on the heat flux emanating from the boundary $x = 0$; i.e., the time dependent quantity
\begin{equation} \label{eq:motexam2}
y(t) = k  \partial_x u(t,0),~t>0.
\end{equation}
The available approaches for finding bounds on \eqref{eq:motexam2} rely on methods for approximating the solution to~\eqref{eq:exampsys} and then computing \eqref{eq:motexam2}. In addition, some existing methods  require  convexity  of the output functional $y(t)$. 

\subsection{Problem Formulation:}

Consider the class of PDE systems governed by
\begin{eqnarray} 
\partial_t u(t,x) &=& F(t,x,D^\alpha u(t,x)), \quad x \in \Omega,~t>0 \label{eq1} \\
y(t) &=& \mathcal{G}u,~t\ge 0 \label{eq:functional}
\end{eqnarray}
subject to $u(0,x) = u_0(x)$ and boundary conditions given by
\begin{equation} \label{eqbc}
Q \begin{bmatrix} D^{\alpha-1}u(t,1) \\ D^{\alpha-1}u(t,0) \end{bmatrix} = 0
\end{equation}
with $Q$ being a matrix of appropriate dimension and \mbox{$F \in \mathcal{R}[t,x,D^\alpha u]$}. 
We assume $\Omega = [0,1]$\footnote{Remark that any bounded domain on the real line can be mapped to [0,1] using an appropriate change of variables.}. The output functional~\eqref{eq:functional} is defined by the operator  $\mathcal{G}$ which is of the form
\begin{multline}
\mathcal{G} u = G_1\left(t,D^\beta u(t,x)\right) \\
+ \int_0^t G_2\left(\tau,D^\beta u(\tau,x)\right) \,\,\mathrm{d}\tau,~x\in \overline{\Omega},~t>0, \label{eqG}
\end{multline}
wherein, $\{G_i\}_{i=1,2}$ are given by
\begin{multline}\label{eg1}
G_i(t,D^\beta u) =  g_1(t,x,D^\beta u(t,x)) \\
+ \int_{\tilde{\Omega}} g_2(t,\theta,D^\beta u(t,\theta)) \,\, \mathrm{d}\theta,\\ ~x \in \overline{\Omega},~t>0,~i=1,2
\end{multline}
with $g_{i} \in \mathcal{R}[t,x,D^\beta u],~i=1,2$ and $\tilde{\Omega}\subseteq \Omega$. In this study, we discuss the cases where either $G_1=0$ or $G_2=0$. The functional given by~\eqref{eq:functional},~\eqref{eqG}, and~\eqref{eg1} represents an output functional either evaluated  
\begin{itemize}
\item [$A.$] at a single point inside the domain ($g_2=0$),
\item [$B.$] over a subset of the domain ($g_1=0$ and $\tilde{\Omega}\subset \Omega$)
\item [$C.$] over the whole domain ($g_1=0$ and $\tilde{\Omega}=\Omega$).
\end{itemize}


The problem we want to solve can be stated as follows. 

\begin{problem}\label{prob:prob1} Given PDE~\eqref{eq1} with initial condition~$u_0 \in \mathcal{U}_0$ and boundary conditions~\eqref{eqbc}, and a scalar $T \ge 0$, compute $\gamma \in \mathbb{R}$ such that $y(T) \le \gamma$, where $y$ is given in~\eqref{eq:functional}.
\end{problem} 


\section{Integral Inequalities} \label{sec:integinq}

We propose a method to solve Problem~\ref{prob:prob1} which requires the solution of integral inequalities. This section briefly presents the results of~\cite{VAP14}, in which, conditions for the verification of integral inequalities, defined in a bounded interval, were proposed. These conditions are obtained by considering a quadratic-like representation of the integrand and  differential relations among the dependent variables. As a result, the positivity of the integral is checked via the positivity of a matrix function, describing the quadratic form in the integrand, over the domain of integration. The conditions and the main steps for their derivation are presented below. 

Consider the following inequality
\begin{multline}\label{eq:intineq}
\mathcal{F} = \int_0^1  (D^{\alpha}u)^\prime F(t,x) (D^{\alpha}u) \,\, \mathrm{d}x \\
-\left[(D^{\alpha-1}u(t,1))^\prime F_1(t) (D^{\alpha-1}u(t,1))\right. \\~~~~~~~~~~~~~~~~ \left.- (D^{\alpha-1}u(t,0))^\prime F_0(t) (D^{\alpha-1}u(t,0))\right]\geq 0 .
\end{multline}
with $F: \real_{\geq 0 } \times [0,1]  \rightarrow \mathbb{S}^{n_\alpha}$, $n_\alpha = \sum_{i = 1}^{n} \alpha_i$, \mbox{$F_i(t): \real_{\geq 0 }   \rightarrow \mathbb{S}^{n_{\alpha-1}}$}, $n_{\alpha-1} = \sum_{i = 1}^{n} (\alpha_i - 1)$, $i = 0,1$  and  the dependent variable $u$ satisfies
\begin{equation}\label{eq:usubspace}
u \in \mathcal{U}_s(Q):=\left\lbrace u \mid Q \left[ \begin{array}{c} D^{\alpha-1}u(t,1) \\ D^{\alpha-1}u(t,0) \end{array} \right] = 0 \right\rbrace.
\end{equation}

In the following, we show how to account for~\eqref{eq:usubspace} when solving~\eqref{eq:intineq}. The lemma below establishes a relation between the values at the boundary $u(t,1)$ and $u(t,0)$ and the integrand and is a straightforward application of the Fundamental Theorem of Calculus. It will be used to introduce extra terms in the integral in~\eqref{eq:intineq}.

\begin{lem}
Consider a matrix function $H(t,x) \in \mathcal{C}^1[x]$, $H : \real_{\geq 0} \times[0,1] \rightarrow \mathbb{S}^{n_{\alpha-1}}$. We have
\begin{multline}
\label{eq:FTC}
\int_0^1   \frac{\mathrm{d}}{\mathrm{d}x} \left[ (D^{\alpha-1}u)^\prime H(t,x) (D^{\alpha-1}u) \right]\,\, \mathrm{d}x  \\
\begin{array}{ll}
= &\int_0^1   (D^{\alpha-1}u)^\prime   \frac{\partial H(t,x)}{\partial x} (D^{\alpha-1}u) \\& ~~~~~~+2  (D^{\alpha-1}u)^\prime  H(t,x) (D^{\alpha}u)   \,\, \mathrm{d}x \\
= &(D^{\alpha-1}u(t,1))^\prime H(t,1) (D^{\alpha-1}u(t,1)) \\ &~~~~~~ - (D^{\alpha-1}u(t,0))^\prime H(t,0) (D^{\alpha-1}u(t,0)).
\end{array}
\end{multline}
\end{lem}
\smallskip

In order to write terms in \eqref{eq:FTC} in a compact form, define the matrix function $\bar{H}(x) \in \mathcal{C}^1[x]$, $\bar{H} : \real_{\geq 0} \times[0,1] \rightarrow \mathbb{S}^{n_{\alpha}}$ to be the matrix satisfying
\begin{multline} \label{eq:Hbar}
(D^{\alpha}u)^\prime \bar{H}(t,x) (D^{\alpha}u)  \\ :=  (D^{\alpha-1}u)^\prime \left[ \frac{\partial H(t,x)}{\partial x} (D^{\alpha-1}u)+2  H(t,x) (D^{\alpha}u) \right].
\end{multline}
Therefore, \eqref{eq:FTC} gives
\begin{multline}
0 =  \int_0^1    (D^{\alpha}u)^\prime \bar{H}(t,x) (D^{\alpha}u)    \,\, \mathrm{d}x  \\
-\left[(D^{\alpha-1}u(t,1))^\prime H(t,1) (D^{\alpha-1}u(t,1))\right. \\ \left.- (D^{\alpha-1}u(t,0))^\prime H(t,0) (D^{\alpha-1}u(t,0))\right],
\end{multline}
which can be added to \eqref{eq:intineq} to give
\begin{multline}
\label{eq:Fmod}
\mathcal{F} = \int_0^1   (D^{\alpha}u)^\prime \left[ F(t,x) + \bar{H}(t,x)\right] (D^{\alpha}u) \,\, \mathrm{d}x \\ -\left[(D^{\alpha-1}u(t,1))^\prime  \left( H(t,1) + F_1(t) \right) (D^{\alpha-1}u(t,1))\right. \\ \left.- (D^{\alpha-1}u(t,0))^\prime \left( H(t,0) + F_0(t)\right) (D^{\alpha-1}u(t,0))\right].
\end{multline}
With the above expression we can then formulate conditions to verify inequality \eqref{eq:intineq} for $u$ satisfying \eqref{eq:usubspace} as follows. Let~${T}\in \mathbb{R}_{\ge0}$.

\begin{prop}
If
\begin{equation}
\label{eq:propMatrix}
F(t,x)+\bar{H}(t,x) \ge 0,~\forall t \in [0,{T}],~x \in [0,1],
\end{equation}
and
\begin{multline}
\label{eq:propBoundaries}
(D^{\alpha-1}u(t,1))^\prime \left( H(t,1) + F_1(t) \right)  (D^{\alpha-1}u(t,1)) \\- (D^{\alpha-1}u(t,0))^\prime \left( H(t,0) + F_0(t)\right) (D^{\alpha-1}u(t,0)) \le 0, \\
\forall u \in \mathcal{U}_s(Q)
\end{multline}
then $\mathcal{F} \ge 0$ for all $u \in \mathcal{U}_s(Q)$ and $t \in  [0,{T}]$.
\end{prop}
\begin{proof}
Refer to~\cite{VAP14}.
\end{proof}
\begin{rem}
As outlined in the beginning of this section, the above results convert the test of~\eqref{eq:intineq} into the test of positivity of the matrix $F(t,x)+\bar{H}(t,x)$ over the domain $x \in [0,1]$ for all $t \in [0,T]$. Moreover, the test is performed for the set of dependent variables belonging to a subspace of a Hilbert space defined by $\mathcal{U}_s(Q)$ as in \eqref{eq:usubspace}. Notice that~\eqref{eq:propMatrix} and \eqref{eq:propBoundaries} are related via matrix $H$ (which defines the entries of $\bar{H}$).
\end{rem}

We transform  output functionals $A$-$B$ to the output functional structure $C$, which we refer as \textit{full integral form} in the sequel. This structure is  consistent with the method for solving integral inequalities outlined in this section. The transformation methods are discussed in Appendix~\ref{sec:full}.

\section{Barrier Functionals} \label{sec:barrier}

 We first recall some results on barrier certificates for ODE systems. Consider the following ODE system
 \begin{eqnarray} \label{eq:odesys}
 \dot{x} = f(t,x), \quad t>0,~x \in \mathcal{X} \subset \mathbb{R}^n,
 \end{eqnarray}
 subject to $x(0)=x_0 \in \mathcal{X}_0 \subset \mathcal{X}$, where \mbox{$f:[0,\infty)\times \mathbb{R}^n \to \mathbb{R}^n$}. The (unsafe) set at time $T$ is denoted by $\mathcal{X}_T \subset \mathcal{X}$.

 \begin{thm}[Theorem 2 in \cite{P06}]
 Let $\mathcal{X}_0,\mathcal{X}_T \subset \mathcal{X}$, and $T>0$. Consider the ODE system described by \eqref{eq:odesys}. If there exists a function $B(t,x)\in \mathcal{C}^1[t,x]$ such that the following conditions hold
 \begin{multline}
 B(T,x(T))-B(0,x_0)>0, \\ \forall x(T)\in \mathcal{X}_T,~\forall x_0\in \mathcal{X}_0,
 \end{multline}
 \begin{multline}
(\partial_x B) f(t,x)+\partial_t B \le 0, \quad \forall t\in [0,T],~\forall x\in \mathcal{X},
 \end{multline}
then there is no solution $x(t)$ of \eqref{eq:odesys} such that $x(0) \in \mathcal{X}_0$ and $x(T)\in \mathcal{X}_T$.
 \end{thm}
\begin{figure}
  \centering
\begin{tikzpicture}[scale=0.7, every node/.style={scale=0.7}]
\path[fill=blue!10] (0,0)--(0,1.5) to [out=10,in=220] (8,5.5) --(8,0);
\draw[very thick,<->] (8,0) node[font=\fontsize{13}{58}\sffamily\bfseries,right]{$t$} -- (0,0) node[font=\fontsize{13}{58}\sffamily\bfseries,left ]{$0$} -- (0,7.5) node[font=\fontsize{13}{58}\sffamily\bfseries,above]{$x(t)$};
\draw[dashed,red, ultra thick] (0,2) to [out=1,in=180] (6,6) --(8,6);
\draw[ultra thick,blue] (4,6)--(4,7.5) node[font=\fontsize{12}{58}\sffamily\bfseries,left ]{$\mathcal{X}_T$};
\draw[very thick,loosely dashed] (4,6)--(4,0) node[font=\fontsize{9}{58}\sffamily\bfseries,below ]{$T$};
\draw[ultra thick, blue] (0,0) --(0,1.5)node[font=\fontsize{12}{58}\sffamily\bfseries,left ]{$\mathcal{X}_0$};
\draw [red] (6,5) node[font=\fontsize{9}{58}\sffamily\bfseries,below ]{$B(T,x)<B(0,x)$};
\draw[red] (2,5) node[font=\fontsize{9}{58}\sffamily\bfseries,above ]{$B(T,x)>B(0,x)$};
\end{tikzpicture}
  \caption{Illustration of the barrier function for ODE systems. For any solution $x(t)$ starting in the set $\mathcal{X}_0$ (shown by shaded blue color), $x(T) \notin \mathcal{X}_T$. }\label{tikzfig1}
\end{figure}
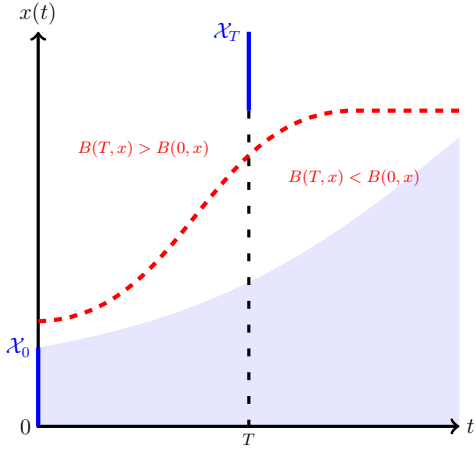

 \begin{rem}
 The level sets of $B(T,x)$ for fixed \mbox{$T>0$} represent barrier surfaces in the $\mathcal{X}$ space separating $\mathcal{X}_0$ and $\mathcal{X}_T$ such that no trajectory of \eqref{eq:odesys} starting from $\mathcal{X}_0$ enters $\mathcal{X}_T$ at time~$T$. This is illustrated in Figure~\ref{tikzfig1} for a single state ODE system.
 \end{rem}

 For PDE systems,  we are interested in finding barrier certificates to check whether the output functional  $y$ as in~\eqref{eq:functional} satisfies $y(T) \le \gamma$ for some $\gamma >0$ and $T>0$, e.g., $y(T)=k\partial_x u(T,0)$ in the motivating example of Section~\ref{sec:mot}. Let $\mathcal{U}_T = \left \{ u \mid y(T) > \gamma\right\}$. The set $\mathcal{U}_T$ defines a subset of function spaces.  At this point, we observe that checking whether $y(T) \le \gamma$ can be performed via an invalidation or safety verification method. The key step is to find certificates that there is no solution $u(t,x)$ to \eqref{eq1} starting at $u_0(x) \in \mathcal{U}_0$ such that $u(T,x) \in \mathcal{U}_T$. The next theorem asserts that  barrier functionals can be used as certificates for upper bounds on output functionals.


\begin{thm} \label{thm1}
Consider the PDE system described by~\eqref{eq1} subject to boundary conditions~\eqref{eqbc} and initial condition $u_0(x)\in \mathcal{U}_0 \subset \mathcal{U} \subseteq \mathcal{
U}_S(Q)$, where $\mathcal{U}_S(Q)$ is defined in~\eqref{eq:usubspace}.  Assume $u \in \mathcal{U} \subseteq \mathcal{U}_S(Q)$. Let
\begin{multline} \label{eq:unsafeset}
\mathcal{U}_T = \bigg\{ u \in  \mathcal{U} \mid \\ y(T)=\int_0^1 g(T,x,D^\beta u(T,x)) \,\, \mathrm{d}x > \gamma \bigg\},\end{multline}
with $\beta>0$, define the unsafe set. If there exists a barrier functional $B(t,D^\beta u) \in \mathcal{C}^1[t,D^\beta u]$,
 such that the following conditions hold
\begin{multline} \label{con1}
B(T,D^\beta u(T,x)) - B(0,D^\beta u_0(x)) >0, \\ \forall u(T,x) \in \mathcal{U}_T,~\forall u_0 \in \mathcal{U}_0
\end{multline}
\begin{equation} \label{con2}
(\partial_{ D^\beta u}B) D^\beta F(t,x,D^\alpha u) +\partial_t B \le 0, \quad \forall t\in[0,T],~\forall u \in \mathcal{U},
\end{equation}
then it follows that there is no solution $u(t,x)$ of \eqref{eq1} such that $u(0,x)=u_0(x) \in \mathcal{U}_0$ and $u(T,x) \in \mathcal{U}_T$ for $T>0$. In other words, it holds that~\mbox{$y(T) \le \gamma$}.
\end{thm}
\begin{proof}
The proof is by contradiction. Assume there exists a solution of \eqref{eq1} such that, for some time $T>0$, $u(T,x) \in \mathcal{U}_T$, i.e., $y(T) > \gamma$. Hence, inequality~\eqref{con1} holds.  From \eqref{con2}, it follows that
\begin{equation}
\int_0^T \frac{\mathrm{d}  B}{\mathrm{d}t} \mathrm{d}t =  \int_0^T\bigg((\partial_{ D^\beta u}B) D^\beta F(t,x,D^\alpha u) +\partial_t B \bigg) \mathrm{d}t \le 0.
\end{equation}
Therefore,
\begin{equation}
B(T,D^\beta u(T,x)) - B(0,D^\beta u(0,x)) \le 0,
\end{equation}
which contradicts \eqref{con1}. Hence, $y(T)\le\gamma$. This completes the proof.
\end{proof}

\begin{rem}
The definition of the set $\mathcal{U}_T$ in Theorem~\ref{thm1} can be  different depending on the application. The particular choice for $\mathcal{U}_T$ in~\eqref{eq:unsafeset} is due to the bound estimation problem under study in this research.
\end{rem}

\begin{figure*}[!t]
{
\begin{eqnarray} \label{eee5}
& minimize_{B}  \left[\gamma(T) \right]  & \nonumber \\
&subject~to& \nonumber \\
&B(t,D^\beta u(t,x)) - B(0,D^\beta u_0) > 0, \quad \forall u_0 \in \mathcal{U}_{0},~\forall u \in \mathcal{U}_T,~t=T, \nonumber \\
& \left(\partial_{D^\beta u}B\right)\left(\partial_t D^\beta u\right)+\partial_t B \le 0, \quad \forall u \in \mathcal{U},~\forall t \in [0,T], &
\end{eqnarray}
\hrulefill
\vspace*{4pt}}
\end{figure*}

\begin{rem}
From Theorem~\ref{thm1}, we can compute upper bounds on $y(T)$ by solving the  minimization problem \eqref{eee5},
where $\mathcal{U}_T$ is given by~\eqref{eq:unsafeset}.
\end{rem}


%
%

Thus far, output functionals of  type~\eqref{eqG} with $G_2=0$ were considered.
In some applications, one might be interested in output functionals of type~\eqref{eqG} with $G_1=0$.
For example, referring to the motivating example in Section~\ref{sec:mot}, we might be interested in the following quantity which represents the average temperature of the heated rod for time $T>0$
$$
y(T) = \int_0^T \int_{\Omega} u(t,x) \,\, \mathrm{d}x \mathrm{d}t.
$$
In other words, inequalities of the following type are sought
\begin{equation} \label{eq:of2}
y(T)= \int_{0}^{T} \int_0^1 g(t,x,D^\beta u(t,x))\,\, \mathrm{d}x\mathrm{d}t \le \gamma^*.
\end{equation}
Obtaining bounds for this type of output functionals can also be addressed  as delineated in the next corollary.

\begin{cor}\label{cor1}
Consider the PDE system described by~\eqref{eq1} with boundary conditions~\eqref{eqbc} and initial condition \mbox{$u_0(x)\in \mathcal{U}_0 \subset \mathcal{U} \subseteq \mathcal{
U}_S(Q)$}, where $\mathcal{U}_S(Q)$ is defined in~\eqref{eq:usubspace}. Assume $u \in \mathcal{U} \subseteq \mathcal{U}_S(Q)$. Let
\begin{multline} \label{eq:unsafeset2}
\mathcal{U}_{[0,T]} = \bigg\{ (t,u) \in [0,T]\times \mathcal{U} \mid \\ \int_0^1 g(t,x,D^\beta u(t,x))\,\, \mathrm{d}x > \partial_t \gamma(t)\bigg\} ,\end{multline}
 with $\beta>0$, define the unsafe set. If there exists a barrier functional $B(t,D^\beta u) \in \mathcal{C}^1[t,D^\beta u]$,
 such that
  \begin{multline} \label{con3}
B(t,D^\beta u(t,x)) - B(0,D^\beta u_0(x)) >0, \\ \forall u \in \mathcal{U}_{[0,T]},~\forall u_0 \in \mathcal{U}_0,~\forall t\in[0,T],
\end{multline}
 and \eqref{con2} are satisfied, then it follows that there is no solution $u(t,x)$ of \eqref{eq1} such that $u(0,x)=u_0(x) \in \mathcal{U}_0$ and \mbox{$u(t,x) \in \mathcal{U}_{[0,T]}$} for $t \in [0,T]$. Hence,  it holds that~\mbox{$y(T) \le \gamma^\star$} with $y(T)$ given by \eqref{eq:of2} and $\gamma^\star = \gamma(T)-\gamma(0)$.
\end{cor}
\begin{proof}
This is a consequence of Theorem~\ref{thm1}. If there exists a function $B(t,D^\beta u(t,x))$ satisfying~\eqref{con3} and~\eqref{con2}, then, from Theorem~\ref{thm1}, we conclude that there is no solution $u(t,x)$ of~\eqref{eq1} satisfying $u(t,x) \in \mathcal{U}_{[0,T]}$ for $t\in[0,T]$. That is, it holds that
\begin{equation} \label{cccc} \int_0^1 g(t,x,D^\beta u(t,x))\,\, \mathrm{d}x \le\partial_t \gamma(t),~\forall t \in [0,T]. \end{equation}
 Integrating both sides of~\eqref{cccc} from $0$ to $T$ yields
\begin{multline}
y(T)=\int_{0}^{T} \int_0^1 g(t,x,D^\beta u(t,x))\,\, \mathrm{d}x \mathrm{d}t \\\le \int_{0}^{T} \partial_t \gamma(t)\,\, \mathrm{d}t  = \gamma(T)-\gamma(0).
\end{multline}
This completes the proof.
\end{proof}

\begin{rem}
We can compute bounds on $\gamma^* = \gamma(T)-\gamma(0)$ via an optimization problem as follows.  If there exists a solution $\gamma^* = \gamma(T)-\gamma(0) $ to the minimization problem \eqref{eee1},
\begin{figure*}[!t]
{
\begin{eqnarray} \label{eee1}
& {minimize}_{B} \left[ \gamma(T)-\gamma(0) \right] & \nonumber \\
&subject~to&\nonumber \\
&B(t,D^\beta u(t,x)) - B(0,D^\beta u_0) > 0, \quad \forall u_0 \in \mathcal{U}_{0},~\forall u \in \mathcal{U}_{[0,T]},~~\forall t \in [0,T], \nonumber \\
& \left(\partial_{D^\beta u}B\right)\left(\partial_t D^\beta u\right)+\partial_t B \le 0, \quad \forall u \in \mathcal{U},~\forall t \in [0,T], &
\end{eqnarray}
\hrulefill
\vspace*{4pt}}
\end{figure*}
 then the following inequality holds
 \begin{equation}
 \int_{0}^{T} \int_0^1 g(t,x,D^\beta u(t,x))\,\, \mathrm{d}x \mathrm{d}t \le  \gamma^*.
 \end{equation}
 \end{rem}


\begin{rem} \label{remNEW}
Notice that in the optimization problem \eqref{eee1}, the \emph{unsafe set} is a problem variable and is parametrized for each time $t\in [0,T]$ according to \eqref{eq:unsafeset2}. The resulting function $B$ may not be a barrier for set 
\begin{equation*} \label{eq:of2NEW}
\mathcal{U} = \left\lbrace u\in \mathcal{U}_S(Q) \mid \int_{0}^{T} \int_0^1 g(t,x,D^\beta u(t,x))\,\, \mathrm{d}x\mathrm{d}t \le \gamma^* \right \rbrace.
\end{equation*}
However, the set described in~\eqref{eq:unsafeset2} can be used to compute the bound as in~\eqref{eq:of2}.
\end{rem}

In order to formulate conditions of Theorem~\ref{thm1} and Corollary~\ref{cor1} in terms of integral inequalities, we consider the following structure for barrier functionals
\begin{equation} \label{eq:bar}
B(t,D^\beta u) = \int_0^1 b(t,x,D^\beta u)\,\, \mathrm{d}x.
\end{equation}
where  $b\in \mathcal{R}[t,x,D^\beta u]$.

\begin{rem} \label{rem133}
The order of partial derivatives of the dependent variables with respect to $x$ in $b(t,x,D^\beta u)$ should be the same as the output functional $y$. This is due to the fact that the barrier functionals serve as barriers in the function space defined by the output functionals. For instance, for the output functional $y(t) = \int_0^1 \left( u^2(t,x) + \left(\partial_x^2 u(t,x) \right)^2 \right)\,\, \mathrm{d}x$, the barrier functional should be of order $2$ in $u$. \end{rem}


\section{Example} \label{sec:exam}

 In this section, we describe how to implement the proposed results using SOS programming by a simple  example:
 
 \begin{itemize}
 \item First, the output functional under study is transformed into the full integral form (Appendix~\ref{sec:full}).
 \item Second, depending on the type of output functionals, the unsafe set is defined as either~\eqref{eq:unsafeset} or~\eqref{eq:unsafeset2}.
\item  Finally, the barrier functional of the appropriate structure is used to find bounds on the output functionals (Remark~\ref{rem133}).
 \end{itemize}
 
 \subsection{SOS Formulation}

Consider~\eqref{eq:exampsys} and output functional~\eqref{eq:motexam2}. Let  \mbox{$f(t,x,u)=f(u)$}  and $k=1$, i.e.
\begin{eqnarray} 
\partial_t u &=& \partial_x^2 u + f(u), \quad x \in [0,1],~t>0 \label{eq:heatexample} \\
y(T)&=&\partial_x u(T,0),~T>0
\end{eqnarray}
subject to  $u(0,x)=u_0(x)$ and $Q \begin{bmatrix} u(t,1) \\ u(t,0) \end{bmatrix}=0$. We are interested in bounding  $y(T)$. Let us transform the output functional to the full integral form using the methods given in Appendix~\ref{sec:full}. From \eqref{e1}, it follows that
$$
y(T) = \frac{-1}{p(0)} \int_0^1  \bigg((\partial_x p(x))\partial_x u(T,x)+p(x)\partial_x^2 u(T,x)\bigg)\,\, \mathrm{d}x,
$$
for some polynomial $p$ such that $p(1)=0$. Setting \mbox{$p(0)=-1$} yields
$$
y(T) =  \int_0^1  \bigg((\partial_x p(x))\partial_x u(T,x)+p(x)\partial_x^2 u(T,x)\bigg)\,\, \mathrm{d}x.
$$
which is a full integral form for the output functional $\partial_x u(T,0)$.
As the next step, we seek certificates showing that no solution belongs to 
\begin{multline}
\mathcal{U}_T = \bigg\{ u \in \mathcal{U}_S(Q) \mid  \int_0^1  \bigg((\partial_x p(x))\partial_x u(T,x)\\+p(x)\partial_x^2 u(T,x) -\gamma \bigg)\,\, \mathrm{d}x > 0 \bigg\} \nonumber\end{multline} at time $T>0$. Applying  Theorem 1 in~\cite{1184414}, for fixed $\gamma$ and $p(x)$, Theorem~\ref{thm1} can be reformulated as follows. If there exist a function $b(t,x,D^1 u)$ such that
\begin{multline} \label{e6}
b(T,x,D^1u(T,x)) - b(0,x,D^1u_0(x)) \\- l_1(x, D^2u(T,x))x(1-x)  \\
- l_2\bigg((\partial_x p(x))\partial_x u(T,x)+p(x)\partial_x^2 u(T,x) -\gamma \bigg)  \\
+ D^2u(T,x) \bar{H}_1(T,x)D^2u(T,x)\in \Sigma\left[x,D^2u(T,x) \right]
\end{multline}
and
\begin{multline} \label{e7}
-\left({\partial_{D^1 u}}b\right) D^1(\partial_x^2 u+f(u)) - \partial_t b
\\- l_3(x,t,D^3 u)t(T-t)  -l_4(x,t,D^3 u)x(1-x) \\ + D^3u \bar{H}_2(t,x)D^3u\in \Sigma\left[x,t,D^3u\right]
\end{multline}
for some $l_1,l_3,l_4 \in \Sigma$, $l_2>0$ and $\{\bar{H}_i\}_{i=1,2}$ as in~\eqref{eq:Hbar}, then \mbox{$y(T)=\partial_x u(T,0) \le \gamma$}. Also, conditions~\eqref{e6} and~\eqref{e7} correspond to~\eqref{con1} and~\eqref{con2}, respectively.  Notice that for $l_2$ fixed and both $\gamma$ and $p(x)$ as variables, SOS inequalities~\eqref{e6} and~\eqref{e7} are  convex   and one can minimize $\gamma$ subject to~\eqref{e6} and~\eqref{e7} which is the same as the minimization problem~\eqref{eee5}. The SOS formulation for Corollary~1 can be carried out similarly.

\subsection{Numerical Results}

The numerical results given in this section was obtained using  SOSTOOLS v. 3.00~\cite{PAVPSP13} and the resultant SDPs were solved using SeDuMi v.1.02~\cite{Stu98}.

Consider PDE~\eqref{eq:heatexample} with $f(u)=\lambda u$
subject to initial conditions \mbox{$u_0(x) = \pi x(1-x)$} and  boundary conditions $u(t,0)=u(t,1)=0$ yielding $Q=\begin{bmatrix} 1 & 0 & 0 & 0 \\ 0 & 0& 1& 0 \end{bmatrix}$. The system is known to be convergent to the null solution just for \mbox{$\lambda\le \pi^2$}~\mbox{\cite[p. 11]{Str04}}. Here,
for illustration purposes, let \mbox{$\lambda = 10\pi^2$}. Notice that convergence of the solutions of the PDE to the null solution  is not required in the proposed method using barrier functionals.


We investigate the bounds on the heat flux emanating from the  boundary $x=0$ at time $T>0$ given by
$$
y(T) = \partial_x u(T,0).
$$
For $T=0.01$, using the proposed method, we obtained the following bound
$$
y(0.01) \le 3.3418.
$$
The actual heat flux from numerical experiments is $y(0.01)=3.212$.  The obtained barrier functional  is given in  Appendix~\ref{sec:app}.
Next,  we consider the following output functional
 \begin{equation}
y(T) =  \int_0^T \partial_x u(\tau,0)\,\, \mathrm{d}\tau,
 \end{equation}
 with $T=0.1$. Using the method presented in Section \ref{sec:barrier}, the obtained upper bound was
 \begin{equation}
 y(0.1) \le 0.5737.
 \end{equation}
 Whereas, the value obtained through numerical simulation and numerical integration is $y(0.1)=0.5656$. The constructed certificates are given in Appendix~\ref{sec:app}.

\section{CONCLUSIONS AND FUTURE WORK} \label{sec:conclusions}

\subsection{Conclusions}

We proposed a methodology to upper-bound output functionals of a class of PDEs by barrier functionals. We transformed different output functionals to the structure suitable for our analyses through splitting the domain and integration-by-parts. For the case of polynomial dependence on both independent and dependent variables, we used SOS programming to construct the barrier functionals by solving SDPs. The proposed method was illustrated with an example.

\subsection{Future Work}

Numerous applications, e.g.  the drag and lift estimation problem described in Section~\ref{sec:intro}, require studying the output functionals of systems defined in two or three dimensional domains.  Therefore, a formulation analogous to the one discussed in Section~\ref{sec:integinq} for integral inequalities over domains of higher dimension  is required.  Furthermore, for some PDEs, the barrier functionals may be conservative. Hence, one may need to adopt special structures for the barrier functionals (see~\cite{AP09} for a special structure for ODEs). Lastly, the application of barrier functionals is not limited to bounding output functionals. Future research can explore other open problems such as safety verification.

%


\bibliography{references}
\bibliographystyle{IEEEtran}


\appendix
{
 \renewcommand{\theequation}{A.\arabic{equation}}
  \setcounter{equation}{0}  

\subsection{Transformation to full integral form}\label{sec:full}

\subsubsection{Boundaries}
Consider functional~\eqref{eg1} with~$g_2=0$ and~$x \in \{0,1\}$, i.e.
\begin{equation}
{y}(t) = {g}\left(t,0,D^\alpha u(t,0)\right),~x_0\in \partial\Omega.
\end{equation}
 For some $p \in \mathcal{C}^1(\Omega)$ satisfying $p(1)=0$, we obtain
\begin{align}
p(0) {g}\left(t,0,D^\alpha u(t,0)\right) =  - \int_0^1 \partial_x (pg) \,\,\mathrm{d}x .
\end{align}
Therefore,
\begin{multline} \label{e1}
y(t) = {g}\left(t,0,D^\alpha u(t,0)\right)  \\= \frac{-1}{p(0)} \int_0^1  \left((\partial_x p){g}+p(\partial_x g) \right)\,\, \mathrm{d}x.
\end{multline}

In addition, if the functional was defined on the boundary $x=1$, assuming $p(0)=0$, we  obtain
\begin{multline} \label{e2}
y(t) = {g}\left(t,1,D^\alpha u(t,1)\right) \\= \frac{1}{p(1)}\int_0^1 \left((\partial_x p){g}+p(\partial_x g)\right) \,\, \mathrm{d}x.
\end{multline}
Notice that, by fixing the values of $p(0)$ and $p(1)$ in~\eqref{e1} and~\eqref{e2}, respectively, we can use equations~\eqref{e1} and~\eqref{e2} to study functionals evaluated at the boundaries using integral inequalities in the full integral form.

\subsubsection{Single Points Inside the Domain}

At this point, consider  functional~\eqref{eg1} with \mbox{$g_2=0$}, i.e.
\begin{equation} \label{eqfun2}
y(t)={g}\left(t,x_0,D^\beta u(t,x_0)\right),~x_0\in \Omega.
\end{equation}
 We split the domain into two subsets $\Omega_1 = (0,x_0]$ and \mbox{$\Omega_2=[x_0,1)$}.
 Then, PDE~\eqref{eq1} can be represented by the following coupled PDEs
\begin{equation*}
\partial_t u=
\begin{cases}
F(t,x,D^\alpha u), & x \in \Omega_1 \\
F(t,x,D^\alpha u), & x \in \Omega_2
\end{cases}
\end{equation*}
subject to $D^{\alpha-1} u(t,x_0)=D^{\alpha-1}u(t,x_0)$ and~\eqref{eqbc}. Using appropriate change of variables, we obtain
\begin{equation*}
\begin{cases}
\partial_t u_1 = F_1(t,x,D^\alpha u_1), & x \in \Omega \\
\partial_t u_2 = F_2(t,x,D^\alpha u_2), & x \in \Omega
\end{cases}
\end{equation*}
subject to $\frac{1}{x_0^{\alpha-1}}D^{\alpha-1} u_1(t,1) = \frac{1}{(1-x_0)^{\alpha-1}}D^{\alpha-1} u_2(t,0)$\footnote{
To simplify the notation, we define
$$\frac{1}{x_0^{\alpha-1}}D^{\alpha-1} u = \left(u,~\frac{1}{x_0} \partial_x u, \ldots, \frac{1}{x_0^{\alpha-1}} \partial_x^{\alpha-1} u \right)^\prime. \nonumber$$
} 
and
\begin{equation*} \label{eqbc3}
Q \begin{bmatrix} \frac{1}{x_0^{\alpha-1}}D^{\alpha-1}u_2(t,1) \\ \frac{1}{(1-x_0)^{\alpha-1}}D^{\alpha-1}u_1(t,0) \end{bmatrix} = 0,
\end{equation*}
where  $Q$ is as in~\eqref{eqbc}, $F_1 = F(t,x,\frac{1}{x_0^\beta}D^\beta u_1)$, and $F_2 = F(t,x,\frac{1}{(1-x_0)^\beta} D^\beta u_2)$. Then, functional~\eqref{eqfun2} can be changed to either of the following
\begin{eqnarray*}
y(t)&=&{g}\left(t,x_0,\frac{1}{x_0^\beta}D^\beta u_1(t,1)\right), \\
y(t)&=&{g}\left(t,x_0,\frac{1}{(1-x_0)^\beta}D^\beta u_2(t,0)\right),
\end{eqnarray*}
and the method proposed for points at the boundaries described in previous subsection can be used.

%

\subsubsection{Subsets Inside the Domain} \label{sec:subsets}

 Consider functional~\eqref{eg1} with~$g_1=0$, i.e.
\begin{equation}\label{eqfun3}
{y}(t)=\int_{\tilde{\Omega}} {g}\left(t,x,D^\beta u(t,x)\right)\,\, \mathrm{d}x,
\end{equation}
where $\tilde{\Omega}=[x_1,x_2] \subset \Omega$.
Similar to the previous section, we split the domain into three subsets $\Omega_1 = (0,x_1]$, $\Omega_2 = [x_1,x_2]$, and $\Omega_3=[x_2,1)$.
Then, PDE  \eqref{eq1} can be rewritten as
\begin{equation*}
\partial_t u=
\begin{cases}
F(t,x,D^\alpha u), & x \in \Omega_1 \\
F(t,x,D^\alpha u), & x \in \Omega_2 \\
F(t,x,D^\alpha u), & x \in \Omega_3,
\end{cases}
\end{equation*}
subject to $D^{\alpha-1} u(t,x_1)=D^{\alpha-1} u(t,x_1)$, $D^{\alpha-1} u(t,x_2)=D^{\alpha-1} u(t,x_2)$, and~\eqref{eqbc}.  With appropriate change of variables, we have
\begin{equation*}
\begin{cases}
\partial_t u_1 = F_1(t,x,D^\alpha u_1), & x \in \Omega \\
\partial_t u_2 = F_2(t,x,D^\alpha u_2), & x \in \Omega \\
\partial_t u_3 = F_3(t,x,D^\alpha u_3), & x \in \Omega
\end{cases}
\end{equation*}
subject to $\frac{1}{(x_1)^{\alpha-1}}D^{\alpha-1} u_1(t,1) = \frac{1}{(x_2-x_1)^{\alpha-1}}D^{\alpha-1} u_2(t,0)$  and $\frac{1}{(x_2-x_1)^{\alpha-1}} D^{\alpha-1} u_2(t,1) = \frac{1}{(1-x_2)^{\alpha-1}}  D^{\alpha-1} u_3(t,0)$ in addition to
\begin{equation*} \label{eqbc2}
Q \begin{bmatrix} \frac{1}{(1-x_2)^{\alpha-1}} D^{\alpha-1}u_3(t,1) \\ \frac{1}{(x_1)^{\alpha-1}} D^{\alpha-1}u_1(t,0) \end{bmatrix} = 0,
\end{equation*}
where  $Q$ is the same matrix as the one in~\eqref{eqbc}, \mbox{$F_1 = F(t,x,\frac{1}{x_1^\beta}D^\beta u_1)$},  $F_2 = F(t,x,\frac{1}{(x_2-x_1)^\beta} D^\beta u_2)$, and $F_3 = F(t,x,\frac{1}{(1-x_2)^\beta} D^\beta u_3)$. Finally, functional \eqref{eqfun3} can be converted to the following full integral form which is suitable for the integral inequalities
\begin{multline*}
{y}(t)= \\(x_2-x_1)\int_{0}^{1} {g}\left(t,x,\frac{1}{(x_2-x_1)^\beta}D^\beta u_2(t,x)\right)\,\, \mathrm{d}x.
\end{multline*}

\subsection{Obtained Certificates} \label{sec:app}
The obtained barrier functional for bounding \mbox{$y(0.01) = \partial_x u(0.01,0)$}:
\begin{eqnarray}
B(t,D^1u) = \int_0^1 b(t,x,D^1u) \,\, \mathrm{d}x \label{appbar}
\end{eqnarray}
\begin{multline}
b(t,x,D^1u) = - 7.1441t^2u^2 + 1.7154t^2u\partial_x u - 20.228t^2u \nonumber \\
 - 7.5293t^2(\partial_x u)^2 - 3.0302t^2\partial_x u + 84.477t^2 \nonumber \\
 + 5.306txu^2
  + 4.439tu^2 - 11.394txu\partial_x u  \nonumber \\ + 4.0763tu\partial_x u + 11.385tux
   + 9.753tu \nonumber \\- 0.7447tx(\partial_x u)^2 + 0.55552t(\partial_x u)^2 - 4.2529tx\partial_x u \nonumber \\
    + 1.3549t\partial_x u - 42.631tx - 28.656t \nonumber \\- 6.9887x^2u^2 + 5.7104xu^2
   - 2.3317u^2 \nonumber \\ - 0.21259x^2u\partial_x u + 2.3274xu\partial_x u - 1.7012u\partial_x u \nonumber \\
    + 5.7105x^2u - 6.7309xu - 0.23359u \nonumber \\- 0.3866x^2(\partial_x u)^2
     + 0.326x(\partial_x u)^2 - 0.048049(\partial_x u)^2 \nonumber \\ - 0.01152x^2\partial_x u + 0.062023x\partial_x u      - 0.020951\partial_x u,
\end{multline}\normalsize
with polynomial $p(x)$ in \eqref{e1} computed as
\begin{align*}
p(x)=0.9999x -0.9999.
\end{align*}

The obtained certificates for bounding
\mbox{$y(0.1) =  \int_0^{0.1} \partial_x u(\tau,0)\,\, \mathrm{d}\tau$}:

 \begin{multline*}
 \gamma(t) = 964.11t^7 + 6.7729t^6 + 66.924t^5 \\+ 32.375t^4 + 100.79t^3 - 4.5509t^2 + 5.7891t,
\end{multline*}
 and~\eqref{appbar} with
 \begin{multline}
 b(t,x,D^1u) = - 1.6454t^2u^2 + 0.37053t^2u\partial_x u - 2.14t^2u  \nonumber \\ - 1.2514t^2(\partial_x u)^2
  + 0.15851t^2\partial_x u + 3.8517t^2 \nonumber \\
  + 2.6005txu^2 + 1.672tu^2
  - 2.5321tu(\partial_x u)^2 \nonumber \\ +0.69396tu\partial_x u + 3.6274txu + 1.3629tu \nonumber \\
   - 0.13896tx(\partial_x u)^2 + 0.24091t(\partial_x u)^2 - 0.18622tx\partial_x u \nonumber \\+ 0.016255t\partial_x u
    - 7.0307tx - 2.1521t \nonumber \\ -6.3404x^2u^2 + 4.1574xu^2     - 1.9985u^2  \nonumber \\+ 0.19901x^2u\partial_x u + 0.54298xu\partial_x u - 0.69336u\partial_x u \nonumber \\
      + 0.75643x^2u - 1.1887xu - 0.048215u \nonumber \\
      - 0.20306x^2(\partial_x u)^2
       + 0.16273x(\partial_x u)^2 - 0.021415(\partial_x u)^2 \nonumber \\ - 0.0023234x^2\partial_x u + 0.014826x\partial_x u
        - 0.0032264\partial_x u.
 \end{multline}\normalsize
 with polynomial $p(x)$ in \eqref{e1} computed as
\begin{align*}
p(x)=0.5056x^2 + 0.4944x - 1.0.
\end{align*}

\end{document}